\documentclass[a4paper,10pt,twoside]{amsart}
\usepackage{amsthm}
\usepackage{amsmath}
\usepackage{amssymb}
\usepackage{amsfonts}
\usepackage{amscd}
\usepackage[english]{babel}
\usepackage{stmaryrd}
\usepackage{enumerate}

\newtheorem{thm}{Theorem}
\newtheorem{cor}{Corollary}
\newtheorem{prop}{Proposition}
\newtheorem{rem}{Remark}

\newtheorem*{conj}{Conjecture}
\newtheorem*{defi}{Definition}
\newtheorem{lem}{Lemma}

\newtheorem*{prop2}{Proposition}

\begin{document}

    \title{The matrix equation $AX-XA=X^{\alpha}g(X)$\\ over fields or rings}
    \author{Gerald BOURGEOIS}
    
    \date{July-7-2014}
    \address{G\'erald Bourgeois, GAATI, Universit\'e de la polyn\'esie fran\c caise, BP 6570, 98702 FAA'A, Tahiti, Polyn\'esie Fran\c caise.}
    \email{bourgeois.gerald@gmail.com}
        
  \subjclass[2010]{Primary 15A24}
			\keywords{}

\begin{abstract} 
Let $n,\alpha\geq 2$. Let $K$ be an algebraically closed field with characteristic $0$ or greater than $n$. We show that the dimension of the variety of pairs $(A,B)\in {M_n(K)}^2$, with $B$ nilpotent, that satisfy $AB-BA=A^{\alpha}$ or $A^2-2AB+B^2=0$ is $n^2-1$ ; moreover such matrices $(A,B)$ are simultaneously triangularizable. Let $R$ be a reduced ring such that $n!$ is not a zero-divisor and $A$ be a generic matrix over $R$ ; we show that $X=0$ is the sole solution of $AX-XA=X^{\alpha}$. Let $R$ be a commutative ring with unity ; let $A$ be similar to $\mathrm{diag}(\lambda_1I_{n_1},\cdots,\lambda_rI_{n_r})$ such that, for every $i\not= j$, $\lambda_i-\lambda_j$ is not a zero-divisor. If $X$ is a nilpotent solution of $XA-AX=X^{\alpha}g(X)$ where $g$ is a polynomial, then $AX=XA$.
\end{abstract}

\maketitle
\section{Introduction}
$\bullet$ Let $n$ be an integer at least $2$. In the first part, $K$ is assumed to be a field such that its characteristic $\mathrm{char}(K)$ is $0$ or greater than $n$. Let $k$ be an integer at least $2$ and $A,B$ be two $n\times n$ matrices, with entries in $K$, satisfying the matrix equation
\begin{equation}  \label{casgen}  \sum_{j=0}^k(-1)^j\binom{k}{j}A^{k-j}B^j=0_n  \end{equation}
 In the following lines, we use the results of \cite{1}. Firstly, $A$ and $B$ have same spectrum $(\lambda_i)_i$ over $\overline{K}$, the algebraic closure of $K$ ; moreover, for every $i$, the generalized eigenspaces $E_{\lambda_i}(A)$ and $E_{\lambda_i}(B)$ are equal. Thus,  to study the solutions $(A,B)$ of Eq (\ref{casgen}) can be reduced to study the restrictions of $A,B$ to a generalized eigenspace $E_{\lambda}(A)=E_{\lambda}(B)$. Moreover, if $(A,B)$ is a solution of Eq (\ref{casgen}), then, for every $\mu\in K$, $(A-\mu I_n,B-\mu I_n)$ is also a solution of Eq (\ref{casgen}). Finally it suffices to solve Eq (\ref{casgen}) when $A,B$ are assumed to be nilpotent matrices.\\
 Note that $k=2$ is a very special case ; indeed Eq (\ref{casgen}) for $k=2$ is
\begin{equation} \label{carre} A^2-2AB+B^2=0_n \end{equation}
and is equivalent to 
\begin{equation}  \label{sim}  N^2=[N,B]\text{ where }N=A-B.  \end{equation}
 Thus Eq (\ref{carre}) is linked to the equation in the unknown $X$
\begin{equation} \label{two}  AX-XA=X^2. \end{equation}
Two matrices $A,B\in M_n(K)$ are said to be simultaneously triangularizable (abbreviated to $ST$) over $K$ if there exists $P\in GL_n(K)$ such that $P^{-1}AP$ and $P^{-1}BP$ are upper triangular matrices.\\
We show 
\begin{prop}  \label{nilp}
We assume that $\mathrm{char}(K)>n$ or is $0$. If $(A,B)$ is a solution of Eq (\ref{carre}), then $A$ and $B$ are $ST$ over $\overline{K}$.
\end{prop}
Note that the previous result is false with regard to the following (Eq (\ref{casgen}) when $k=3,n=4$)
\begin{equation} \label{equat3} A^3-3A^2B+3AB^2-B^3=0_n. \end{equation}
We consider the relation linking the $n\times n$ matrices $A,B$
\begin{equation} \label{relat}  AB-BA=A^{\alpha} \text{ where }\alpha\geq 2. \end{equation}
We show that the dimension of the algebraic variety of pairs $(A,B)\in M_n(\overline{K})$, with $B$ nilpotent, that satisfy Eq (\ref{carre}) or Eq (\ref{relat}) is $n^2-1$.\\
$\bullet$ In the second part, $R$ is assumed to be a commutative ring and we study the equation
\begin{equation} \label{equatLie}  AX-XA=X^{\alpha} \text{ where }\alpha\geq 2. \end{equation} 
\begin{defi} Let $R$ be a commutative ring with unity and $A=[a_{i,j}]$ be a $n\times n$ matrix where the $(a_{i,j})$ are commuting indeterminates. If $\tilde{R}$ is the ring of the polynomials in the indeterminates $(a_{i,j})$ and with coefficients in $R$, then the algbebra generated by $A$ is in $M_n(\tilde{R})$. In particular, there are no polynomial relations, with coefficients in $R$, linking the $(a_{i,j})_{i,j}$. We say that $A$ is a \emph{generic} matrix over $R$.
\end{defi}
When $R$ is reduced (for every $u\in R$, $u^2=0$ implies $u=0$), we obtain a precise result
\begin{prop}   \label{red}
Let $n\geq 2$, $R$ be a reduced ring such that $n!$ is not a zero-divisor. Let $A\in M_n(\tilde{R})$ be a generic matrix. Then $X=0$ is the sole solution of Eq (\ref{equatLie}).
\end{prop}
Else, we only obtain a partial result
\begin{prop}   \label{triang}
	Let $R$ be a commutative ring with unity, $\alpha\geq 2$ and $g$ be a polynomial with coefficients in $R$ such that $g(0)\not=0$.	Let $X\in M_n(R)$ be a nilpotent solution of the equation
		\begin{equation}  \label{equaring}   XA-AX=X^{\alpha}g(X). \end{equation}	
		Then all elements of the two-sided ideal, in $M_n(R)$, generated by $AX-XA$ are nilpotent. 
		\end{prop}
		If $A$ is diagonalizable and its spectrum is ``good'', then we obtain a complete solution
		\begin{thm}  \label{main}  Let $A\in M_n(R)$ be similar to $\mathrm{diag}(\lambda_1,\cdots,\lambda_n)$ such that, for every $i\not= j$, $\lambda_i-\lambda_j$ is not a zero-divisor. If $n!$ is not a zero-divisor and $X\in M_n(R)$ is a solution of Eq (\ref{equatLie}), then there is $P\in GL_n(R)$ such that $A=P\mathrm{diag}(\lambda_1,\cdots,\lambda_n)P^{-1}$ and $X=P\mathrm{diag}(\mu_1,\cdots,\mu_n)P^{-1}$ where, for every $i$, ${\mu_i}^{\alpha}=0$.				
				\end{thm}
				Finally, we show our main result
					\begin{thm} \label{princ}
Let $A\in M_n(R)$ be similar over $R$ to $\mathrm{diag}(\lambda_1I_{n_1},\cdots,\lambda_rI_{n_r})$ with $n_1+\cdots +n_r=n$ and such that, for every $i\not= j$, $\lambda_i-\lambda_j$ is not a zero-divisor. Let $X$ be a nilpotent solution of Eq (\ref{equaring}). \\
$i)$ Then there is $P\in GL_n(R)$ such that 
$$A=P\;\mathrm{diag}(\lambda_1I_{n_1},\cdots,\lambda_rI_{n_r})\;P^{-1}\text{ and }X=P\;\mathrm{diag}(X_1,\cdots,X_r)\;P^{-1}$$
 where, for every $i$, $X_i\in M_{n_i}(R)$ and ${X_i}^{\alpha}g(X_i)=0$.\\
 $ii)$ If moreover $g(0)$ is a unit, then $X^{\alpha}=0$.
\end{thm}
\section{Equations (2), (5), (7) over a field} 
In this section , $K$ is a field with characteristic not $2$.
Let $J_n$ denote the nilpotent Jordan-block of dimension $n$. If $A$ is a square matrix, then $\chi_A$ denotes its characteristic polynomial.\\
The following result is well-known (see \cite{1})
\begin{prop2}
If $n=2$ and $(A,B)$ is a solution of Eq (\ref{carre}), then $AB=BA$.
\end{prop2}
\begin{cor} \label{mult2}
Let $(A,B)\in {M_n(K)}^2$ be a solution of Eq (\ref{carre}). If the multiplicity of each eigenvalue of $A$ is at most $2$, then $AB=BA$.
\end{cor}
\begin{proof}
According to \cite{1}, we may assume that $A,B$ are nilpotent matrices of dimension $2$ and we conclude using the previous proposition.
\end{proof}
The following result is a slight improvement of \cite[Theorem 1]{5} or of \cite[Theorem 11']{17}. 
\begin{thm}  \label{lie}
We assume that $\mathrm{char}(K)>n$ or is $0$. Let $A,B\in M_n(K)$ be such that $[A,B]=f(A)$ where $f$ is a polynomial. Then $A,B$ are $ST$ over $\overline{K}$.
\end{thm}
\begin{proof}
Let $V$ be the vector space spanned by $\{B,I_n,A,\cdots,A^{n-1}\}$. One checks easily by induction
  \begin{equation} \label{rec}  \text{for all }i\geq 1,\;\; A^iB-BA^i=iA^{i-1}f(A). \end{equation}
By Cayley-Hamilton's Theorem (that is valid on a commutative ring with unity), $A^iB-BA^i$ belongs to $V$, and $V$ is a Lie's algebra. The derived series of $V$ is
$$V_1=[V,V]\subset K[A],\;\;V_2=[V_1,V_1]=\{0\}.$$
Thus $V$ is solvable. According to Lie's Theorem (that is valid when $\mathrm{char}(K)>n$ or is $0$, cf. \cite[p. 38]{6}), $V$ is triangularizable, that is $A,B$ are $ST$.
\end{proof}
We deduce \textbf{Proposition \ref{nilp}}
\begin{proof}
 According to Eq (\ref{sim}) and Theorem \ref{lie}, $N,B$ and consequently, $A,B$ are $ST$. 
\end{proof}
\begin{rem}
$i)$ The hypothesis about $\mathrm{char}(K)$ is necessary ; indeed, if $n=3$ and $\mathrm{char}(K)=3$, then 
$$A=\begin{pmatrix}0&1&0\\0&0&2\\-1&0&0\end{pmatrix}\;,\;B=\begin{pmatrix}0&0&0\\0&0&-1\\1&0&0\end{pmatrix}$$
satisfy $[A,B]=A^2$ and $A^3=I_3$. \\
$ii)$ According to Theorem \ref{lie}, if $(A,B)\in {M_n(K)}^2$ is a solution of 
$$A^2-2AB+B^2=P(A-B),\text{ where }P\text{ is a polynomial,}$$
 then $A$ and $B$ are $ST$ over $\overline{K}$.
\end{rem}
Let $\alpha\in\llbracket 2,n-1\rrbracket$. According to \cite[Proposition 2.9]{9}, if $X\in M_n(\overline{K})$ is a solution of Eq (\ref{equatLie}), then each generalized eigenspace of $A\in M_n(\overline{K})$ is $X$-invariant. Thus we may assume that $A$ is nilpotent.\\
We consider the algebraic varieties 
$$S_{n,\alpha}=\{(A,B)\in {M_n(\overline{K})}^2\;|\;B\text{ is nilpotent and }(A,B)\text{ satisfies Eq (\ref{relat})}\}$$
$$\text{ and }U_n=\{(A,B)\in {M_n(\overline{K})}^2\;|\;B\text{ is nilpotent and }(A,B)\text{ satisfies Eq (\ref{carre})}\}.$$
Recall that the algebraic variety $N_n$ of nilpotent matrices in $M_n(\overline{K})$ has dimension $n^2-n$ and is irreducible (cf. \cite[Section: \emph{The nilpotent cone}]{21}). Note that the algebraic variety
		$$W_n=\{(A,B)\in {M_n(\overline{K})}^2\;|\;A,B\text{ are nilpotent and  }AB=BA \}$$
		has dimension $\dim(N_n)+(n-1)=n^2-1$ and is irreducible when $\mathrm{char}(K)>n$ (cf. \cite{14}).
\begin{prop}    \label{S_n}
The dimension of $S_{n,\alpha}$ is $n^2-1$.
\end{prop}
\begin{proof} A generic nilpotent matrix $B$ is similar to $J_n$.
 Put $B=J_n$ and consider the equation $XJ_n-J_nX=X^{\alpha}$. According to \cite[Remark 3.4]{9}, $X$ is strictly upper triangular and we can express the entries $(x_{i,j})$ of $X$ as functions of $x_{1,2},\cdots,x_{1,n}$. Then the algebraic variety $Y_{n,\alpha}=\{X\;|\;XJ_n-J_nX=X^{\alpha}\}$ has dimension $n-1$. Moreover, if $x_{1,2}$ is chosen non-zero, then $X$ is similar to $J_n$. Thus a component of $S_{n,\alpha}$ of maximal dimension is obtained for generic nilpotent matrices $B$.  We deduce that $\dim(S_{n,\alpha})=\dim(N_n)+\dim(Y_{n,\alpha})=n^2-1$.
\end{proof}
\begin{rem}
$i)$ According to the previous proof, when $n>2$, there are pairs $(A,B)$ of $S_{n,\alpha}$ such that $A$ and $B$ do not commute.\\
$ii)$ Note that $\{(A,B)\in {M_n(\overline{K})}^2\;|\;A,B\text{ are similar to  }J_n\text{ and satisfy Eq (\ref{relat})}\}$ is Zariski open dense in a maximal component of $S_{n,\alpha}$.\\
$iii)$ We may wonder whether $S_{n,\alpha}$ is irreducible when $\mathrm{char}(K)=0$.
\end{rem}
\begin{prop}
The dimension of $U_n$ is $n^2-1$.
\end{prop}
\begin{proof}
Note that $A$ is also nilpotent. According to Proposition \ref{S_n} with $\alpha =2$ and Eq (\ref{sim}), $\dim(U_n)=\dim(S_{n,2})=n^2-1$.
\end{proof}
		\begin{prop}  \label{dim3} We assume that $\mathrm{char}(K)>3$ or is $0$. Let $(A,B)\in M_3(K)$ be a solution of Eq (\ref{carre}) such that $AB\not=BA$. Then there are an invertible matrix $P$ and $\lambda\in K$ such that $P^{-1}AP$ and $P^{-1}BP$ are both in the form $\begin{pmatrix}\lambda&*&*\\0&\lambda&*\\0&0&\lambda\end{pmatrix}$.	Conversely, there exist such solutions.	
		\end{prop}
		\begin{proof}
		According to Corollary \ref{mult2}, necessarily $A$ and $B$ have a sole eigenvalue $\lambda=\dfrac{1}{3}\mathrm{trace}(A)$, that is necessarily in $K$. We conclude with Proposition \ref{nilp}.	An instance of such a solution is $(J_3,\mathrm{diag}(0,\dfrac{1}{2}J_2))$.	
		\end{proof}
		\begin{rem}
		$i)$	We may wonder whether $A,B$ are still $ST$ when $k=3$, that is when $(A,B)$ is a solution of Eq (\ref{equat3}).
							The answer is no, as we can see with the following solution of Eq (\ref{equat3}) when $n=4$	and $\mathrm{char}(K)\geq 5$	or $=0$
		 $$(\begin{pmatrix}0&4/3&-1/3&-1\\1&0&3/4&-3/4\\1&0&0&0\\1&0&0&0\end{pmatrix},J_4).$$
		Clearly $A,B$ are nilpotent and $[A,B]$ is invertible. We say that a pair $(U,V)$ have property L (cf. \cite{13}) if there are orderings of the eigenvalues $(\lambda_i),(\mu_i)$ of $U,V$ such that, for every $a\in K$, the eigenvalues of $U+aV$ are $(\lambda_i+a\mu_i)$ ; if $U,V$ are $ST$, then they have property L. In our instance, $(A,B)$ have not property L because, if $a\not=0$, then $A+aB$ is not nilpotent.\\
		$ii)$ We consider the algebraic variety 
		$$V_4=\{(A,B)\in {M_4(K)}^2\;|\;A,B\text{ are nilpotent and satisfy Eq (\ref{equat3}) }\}.$$
		We use a similar argument to that used in the proof of Proposition \ref{S_n} ; by the Gr\"obner basis method, we solve Eq (\ref{equat3}) with $A=J_4$ ; we obtain an algebraic set of solutions in $B$ that has $6$ as Hilbert dimension. A component of $V_4$ of maximal dimension is obtained for generic nilpotent matrices $A$, that is, for matrices $A$ that are similar to $J_4$. Finally $\dim(V_4)=\dim(N_4)+6=18$.  \\
	A similar calculation shows that the algebraic variety  
		$$V_4\cap W_4=\{(A,B)\in {M_4(K)}^2\;|\;A,B\text{ are nilpotent },AB=BA\text{ and }(A-B)^3=0 \}$$
		has dimension: $\dim(N_4)+2=14$.
				\end{rem}
Let $\alpha,n$ be integers $\geq 2$.
\begin{prop}  \label{zerofield}
Let $K$ be a field such that $\mathrm{char}(K)>n$ or is $0$. If $A\in M_n(K)$ has $n$ distinct eigenvalues in $\overline{K}$, then $X=0$ is the sole solution of Eq (\ref{equatLie}).
\end{prop}
\begin{proof}
Note that $A$ satisfies the property
$$\mathcal{P}:\text{ for every }Y\in M_n(K),\;\{AY=YA,\;Y^{\alpha}=0\}\text{ imply }Y=0.$$ 
Indeed we may assume that $A$ is a diagonal matrix over $\overline{K}$. Since $Y$ commute with $A$, $Y$ is also diagonal and clearly, $Y=0$.
According to Theorem \ref{lie}, $[A,X]$ is nilpotent and $X$ too ; assume that $k$, the nilindex of $X$, is greater than $\alpha$. According to Eq (\ref{rec}), $[A,X^{k-\alpha+1}]=0$ ; by the property $\mathcal{P}$ and $X^{(k-\alpha+1)\alpha}=0$, we deduce $X^{k-\alpha+1}=0$, that is contradictory and therefore $k\leq\alpha$. Thus $X^{\alpha}=0$ and $AX=XA$ ; by Property $\mathcal{P}$, we conclude that $X=0$.
\end{proof}
\begin{rem} The previous result is shown, when $K$ is a field of characteristic $0$, in \cite[Proposition 2.5]{9}.
\end{rem}
				
				\section{Equations (7), (8) over a ring} 
\begin{defi} $i)$  Let $(R_i)_{i\in I}$ be commutative rings with unity. Their ring subdirect product $R$ is defined if
there is $f:R\rightarrow \Pi_{i\in I}R_i$ an injective ring homomorphism such that, for every $j\in I$, the projection of $f$ on $R_j$ is onto.\\
$ii)$ A commutative ring $R$ with unity is reduced if for every $u\in R$, $u^2=0$ implies $u=0$. That is equivalent to $R$ is isomorphic to a subring of a direct product of fields or isomorphic to a subdirect product of domains (cf. \cite[Theorem 11.6.7]{12}).\\
For instance, $R=\mathbb{Z}\times \mathbb{Z}$ is a reduced ring with $\mathrm{char}(R)=0$. More generally, $\mathrm{char}(R)$ is $0$ or a product of distinct primes. Note that $R=\mathbb{Z}/3\mathbb{Z}\times \mathbb{Z}$ is reduced with $\mathrm{char}(R)=0$ and yet, $3=(0,3)$ is a zero-divisor.
\end{defi}
We show \textbf{Proposition \ref{red}}.
\begin{proof}
 Since $R$ is a subring of a direct product of algebraic closed fields $\Pi_{i\in I}K_i$, we may assume $R=\Pi_{i\in I}K_i$ where, for every $i\in I$, $K_i$ is a field such that $\mathrm{char}(K_i)>n$ or is $0$. Let $X=(X_i)_i$ and $A=(A_i)_i$. Thus, for any $i\in I$, $X_iA_i-A_iX_i={X_i}^{\alpha}$ where the $i^{th}$ component $A_i\in M_n(K_i)$ of $A$ is generic ; then, for every $i$, the discriminant of $\chi_{A_i}$ is not $0$ and the matrix $A_i$ has $n$ distinct eigenvalues. According to Proposition \ref{zerofield}, for every $i\in I$, $X_i=0$ and consequently $X=0$.
\end{proof}
\begin{prop} \label{domain} Let $R$ be a commutative ring with unity such that $n!$ is not a zero-divisor and let $A\in M_n(R)$.
If $X\in M_n(R)$ is a solution of Eq (\ref{equatLie}), then $X$ is a nilpotent matrix.
\end{prop}
\begin{proof} 
Note that $[A,X]$ and $X$ commute and that the Cayley-Hamilton theorem is true over $R$. According to the proof of Jacobson lemma (cf. introduction of \cite{2} and also \cite{20} where an improvement of this result is stated within the framework of the algebraic operators on a complex Banach space), $n!\;[A,X]^{2^n-1}=0$, that implies $[A,X]^{2^n-1}=0$ and we are done.
\end{proof}
If $A$ is generic over $R$, then we have a more precise result for small $n,\alpha$.
\begin{prop} Let $n=2$, $2\leq \alpha\leq 4$. Let $R$ be a commutative ring with unity such that, if $\alpha=2,3$ or $4$, then $2,3!$ or $5!$ is not a zero-divisor. Let $A\in M_2(\tilde{R})$ be a generic matrix.  If $X=[x_{i,j}]\in M_2(\tilde{R})$ is a solution of Eq (\ref{equatLie}) 
 then 
$$AX-XA=X^{\alpha}=0\text{ and for every }(i,j), \;{x_{i,j}}^{2\alpha-1}=0.$$
\end{prop}
\begin{proof} 
 The $4$ parameters are the $(a_{i,j})$. We have a system of $4$ equations in the $4$ unknowns $(x_{i,j})$. Using Gr\"obner basis theory in any specified characteristic, we obtain the required result. 
\end{proof}
When $n=3,4,5$ and  $2\leq \alpha\leq 4$, the calculations have great complexity ; thus we carry out specializations of the $(a_{i,j})$ in the ring $R$. Then we randomly choose the matrix $A$ (in order to simulate the generic nature of the matrix) and we formally solve Eq (\ref{equatLie}) in the $n^2$ unknowns $(x_{i,j})$. Numerical experiments, again using Gr\"obner basis theory in characteristic great enough, lead to the following result: $X^{\alpha}=0_n$ and for every $(i,j)$, ${x_{i,j}}^{(\alpha-1)n+1}=0$ ; for instance, if $n=5,\alpha=4$, then the supplementary condition is: ``$13\;!$ is not a zero-divisor''.  Therefore we conjecture
\begin{conj}
Let $n\geq 2$, $R$ be a commutative ring with unity satisfying a condition in the form: ``the integer $\phi(n,\alpha)!$ is not a zero-divisor''. Let $A\in M_n(\tilde{R})$ be a generic matrix. If $X=[x_{i,j}]\in M_n(\tilde{R})$ is a solution of Eq (\ref{equatLie}), then $AX-XA=X^{\alpha}=0_n$ and for every $(i,j)$, ${x_{i,j}}^{(\alpha-1)n+1}=0$.
\end{conj}		
	\begin{rem}
	$i)$ The instance $R=\mathbb{Z}/27\mathbb{Z},U=3I_2$ shows that if $U$ is a nilpotent $n\times n$ matrix, then we have not necessarily $U^n=0$.\\
	$ii)$ In the previous conjecture, note that the exponent $(\alpha-1)n+1$ is very special ; indeed, if $U^{\alpha}=0_n$, then we have 
	$$(trace(U))^{(\alpha-1)n+1}=0,$$
	and we cannot do better (cf. \cite{11}).		
		\end{rem}
					Let $R$ be a commutative ring with unit and $A\in M_n(R)$.
		We look for the $nilpotent$ solutions $X\in M_n(R)$ of Eq (\ref{equaring}),			
		where $\alpha\geq 2$ and $g$ is a polynomial in $X$, with coefficients in $R$, such that $g(0)\not= 0$. Then, according to Eq (\ref{rec}), for every $i$,
		$$X^iA-AX^i=iX^{\alpha+i-1}g(X).$$
		Let $\mathrm{val}(T)$ denote the valuation of the polynomial $T$, with the following convention: $\mathrm{val}(0)=+\infty$. In the sequel, $X$ is a nilpotent solution of Eq (\ref{equaring}).
		\begin{lem}  \label{val1}
		Let $u$ be a polynomial in $X$. Then, for every $l$, $uA^l=\sum_{i=0}^l A^iv_i$ where $v_i$ is a polynomial in $X$ with, for every $i$, $\mathrm{val}(v_i)\geq val(u)$.		
		\end{lem} 
		\begin{proof}
		For every $i$, $X^iA=AX^i+v_0(X)$ with $\mathrm{val}(v_0)\geq i$. Then $uA=Au+v_1(X)$ with $\mathrm{val}(v_1)\geq \mathrm{val}(u)$. In the same way, $uA^2=A^2u+2Av_1+v_2$ and $uA^3=A^3u+3A^2v_1+3Av_2+v_3$ with $\mathrm{val}(v_3)\geq \mathrm{val}(v_2)\geq\mathrm{val}(v_1)$, and so on.
		\end{proof}
		\begin{lem}   \label{val2}
		Let $k\geq 2$ and $(v_i)_i$ be polynomials in $X$ with, for every $i$, $\mathrm{val}(v_i)\geq \alpha$. Then 
		$$(\sum_iA^iv_i)^k\text{ is in the form } \sum_iA^iw_i$$
		$$\text{ where, for every }i,w_i\text{ is a polynomial in }X\text{ with }\mathrm{val}(w_i)\geq k\alpha.$$		
		\end{lem}
		\begin{proof}
		We use a reasoning  by recurrence. Let $E=(\sum_iA^iv_i)^k=(\sum_iA^iz_i)(\sum_iA^iv_i)$ where $z_i$ is a polynomial in $X$ such that $\mathrm{val}(z_i)\geq (k-1)\alpha$. Using Lemma \ref{val1}, $E=\sum_{i,j}A^i(z_iA^j)v_j=\sum A^iA^r y_r v_j$ where $\mathrm{val}(y_r)\geq (k-1)\alpha$.		
		\end{proof}
		\begin{lem} \label{val3}                                                            
		Let $P,Q$ be polynomials in $X,A$. Then, for every $k$,
		$$(P(AX-XA)Q)^k\text{ is in the form }\sum_iA^iv_i$$
		$$\text{ where, for every }i,v_i\text{ is a polynomial in }X\text{such that }\mathrm{val}(v_i)\geq k\alpha.$$		
		\end{lem}
		\begin{proof}
		Let $E=(P(AX-XA)Q)^k=(PX^{\alpha}gQ)^k$. Then $E=\sum\Pi_{i=1}^k(P_iX^{\alpha}gQ_i)$ where, for every $i$, $P_i,Q_i$ are monomials in  the form $A^{i_1}X^{j_1}A^{i_2}X^{j_2}\cdots$. Assume, for instance, that $P_i=A^{i_1}X^{j_1}A^{i_2}X^{j_2}A^{i_3}X^{j_3}$. By Lemma \ref{val1}, $$P_i=A^{i_1}X^{j_1}A^{i_2}(\sum_rA^ru_r)X^{j_3}=\sum A^{i_1}X^{j_1}A^{i_2+r}u_rX^{j_3}=\sum\sum A^{i_1}A^sw_su_sX^{j_3}$$
		and finally $P_i,Q_i$ are in the form $\sum_jA^jp_j,\sum_jA^jq_j$ where $p_j,q_j$ are polynomials in $X$. Now 
		$$E=\sum\Pi_{i=1}^k(\sum_jA^jp_j)X^{\alpha}g(\sum_jA^jq_j)=$$
		$$\sum(A^{i_1}p_1X^{\alpha}gA^{j_1}q_1)\cdots (A^{i_k}p_kX^{\alpha}gA^{j_k}q_k)=\sum(A^{i_1}\tilde{v}_1)\cdots(A^{i_k}\tilde{v}_k),$$
		  where, for every $i$, $\mathrm{val}(\tilde{v}_i)\geq \alpha$. Using Lemma \ref{val2}, $E=\sum_iA^iv_i$ where, for every $i$, $\mathrm{val}(v_i)\geq k\alpha$.		
		\end{proof}
				We deduce \textbf{Proposition \ref{triang}}
		\begin{proof}
		Use Lemma \ref{val2}, Lemma \ref{val3} and the fact that $X$ is a nilpotent matrix.
		\end{proof}
		\begin{rem}
		$i)$ When $R$ is an arbitrary algebraically closed field, the previous result is equivalent to: $A$ and $X$ are $ST$ over $R$ (cf. \cite{16} and compare with Theorem \ref{lie}).\\ 
		$ii)$ When $R$ is a ring, McCoy, in \cite{15}, gave an equivalent condition that, unfortunately, seems almost useless. In fact,  if $AB=BA$ and $A,B$ are triangularizable over $R$, then they have not necessarily a common eigenvector ; the following example, for $n=2$, is due to J. Starr: 
		$$R=\mathbb{C}[\epsilon]/\langle{\epsilon}^2\rangle ,A=\begin{pmatrix}0&\epsilon\\0&0\end{pmatrix},B=\begin{pmatrix}0&0\\\epsilon&0\end{pmatrix}.$$  		
		\end{rem}
		Note that we can reduce (theoretically) the resolution of Eq (\ref{relat}) to the case $\alpha=2$. Indeed, let $(A,B)$ be satisfying Eq (\ref{relat}). According to Eq (\ref{rec}), $A^{\alpha-1}B-BA^{\alpha-1}=(\alpha-1)A^{2\alpha-2}$. Put $B=(\alpha-1)B_1, A_1=A^{\alpha-1}$ ; if $\alpha-1$ is not a zero-divisor, then $A_1B_1-B_1A_1={A_1}^2$.\\		
	\indent	We have a more precise result when $A$ is diagonalizable and its spectrum is ``good''. 
		\begin{lem}   \label{commut}
		 Let $A\in M_n(R)$ be similar to $\mathrm{diag}(\lambda_1,\cdots,\lambda_n)$ and $B\in M_n(R)$ such that $AB=BA$. \\
		$i)$ Assume that, for every $i\not= j$, $\lambda_i-\lambda_j$ is not a zero-divisor in $R$. Then $A$ and $B$ are simultaneously diagonalizable.\\
		$ii)$ Assume that, for every $i\not= j$, $\lambda_i-\lambda_j$ is a unit. Then $B$ is a polynomial in $A$ of degree at most $n-1$ and with coefficients in $R$.
				\end{lem}
				\begin{proof}
	 We may assume that $A=\mathrm{diag}(\lambda_1,\cdots,\lambda_n)$. \\
	$i)$ If $B=[b_{i,j}]$, then $AB-BA=[c_{i,j}]$ with $c_{i,j}=(\lambda_i-\lambda_j)b_{i,j}$ ; if $i\not= j$, then $b_{i,j}=0$.\\
	$ii)$ According to $i)$, we may assume that $B=\mathrm{diag}(\mu_1,\cdots,\mu_n)$. We must solve the linear system, in the unknowns $(\alpha_i)_{0\leq i<n}$:
	$$\text{ for every }j,\;\mu_j=\sum_i \alpha_i{\lambda_j}^i.$$
	Since the determinant of the associated Vandermonde matrix is a unit, we are done. 
				\end{proof}
				\begin{prop}  \label{diago}
			Let $A\in M_n(R)$ be similar to $\mathrm{diag}(\lambda_1,\cdots,\lambda_n)$ such that, for every $i\not= j$, $\lambda_i-\lambda_j$ is not a zero-divisor and $B\in M_n(R)$. If $A$ and $[A,B]$ commute, then $AB=BA$.					
			\end{prop}
			\begin{proof}
			We may assume that $A=\mathrm{diag}(\lambda_1,\cdots,\lambda_n)$ and put $B=[b_{i,j}]$, $D=[d_{i,j}]=A[A,B]-[A,B]A$. We obtain $d_{i,j}=(\lambda_i-\lambda_j)^2b_{i,j}=0$ ; therefore, if $i\not= j$, then $b_{i,j}=0$.			
			\end{proof}
		Now we prove 	\textbf{Theorem \ref{main}}				
				\begin{proof}  
					 According to Proposition \ref{domain}, $X$ is nilpotent.	Assume that the nilindex of $X$ is $i+\alpha-1$. According to Eq (\ref{rec}), $X^i$ and $A$ commute. Since $X^{i-\alpha+1}A-AX^{i-\alpha+1}=(i-\alpha+1)X^i$, $[X^{i-(\alpha-1)},A]$ and $A$ commute.	
				According to Proposition \ref{diago}, $[X^{i-(\alpha-1)},A]=0$. Then we obtain a finite sequence of matrices
			$$(X^{i-k(\alpha-1)})_{0 \leq k< i/(\alpha-1)}$$
			that commute with $A$ and where, at each step, the exponent decreases by $\alpha-1$. Finally we obtain a matrix $X^{\beta}$, with $\beta\in\llbracket 1,\alpha-1\rrbracket$, that commutes with $A$. Since $X^{\beta+\alpha-1}=0$, $[X^{\alpha},A]=\alpha X^{\beta+\alpha-1}X^{\alpha-\beta}=0$. Finally $[X,A]$ and $A$ commute, that implies $[A,X]=0$.	
						\end{proof}
	We show	\textbf{Theorem \ref{princ}}, our main result.					
	\begin{proof} We may assume that $A=\mathrm{diag}(\lambda_1I_{n_1},\cdots,\lambda_rI_{n_r})=\mathrm{diag}(\mu_1,\cdots,\mu_n)$.\\ 
$i)$ $\bullet$ For every $k\geq 1$, $\ker((A-\lambda_1I)^k)=[e_1,\cdots,e_{n_1}]$ where $(e_i)_i$ is the canonical basis of $R$. Indeed, if $x=[x_1,\cdots,x_n]^T$ and $(A-\lambda_1 I)^k x=0$, then, for every $i$, $(\mu_i -\lambda_1)^kx_i=0$. Thus, for every $i>n_1$, $x_i=0$.\\
$\bullet$ $X(\ker(A-\lambda_1I))\subset \ker(A-\lambda_1I)$. Indeed, this comes from the proofs of \cite[Lemmas 1 and 2, Theorem 3]{5}. Therefore $AX=XA$.\\
$\bullet$ In the same way than in the proof of Lemma \ref{commut} i), we obtain that $X$ has the required form.\\
$ii)$  In $M_n(R)$, $g(0)I_n$ is a unit and $g(X)-g(0)I_n$ is nilpotent ; therefore $g(X)$ is a unit and we are done.
\end{proof}			
				\begin{rem}	In the previous proposition, consider a matrix $X_i$. According to \cite[Theorem 8.54]{18}, $\det(g(X_i))$ is the resultant $\mathrm{Res}(\chi_{X_i},g)$. Therefore, if $\mathrm{Res}(\chi_{X_i},g)$ is a unit, then ${X_i}^{\alpha}$ is zero again ; in general, it is not, as we see in the following instance: let $\tau\in R$ such that $\tau^2\not=0,\tau^3=0$, $X_i=\tau I_2, \;g(x)=\tau$ and $\alpha=2$. Then ${X_i}^2\not=0,\,{X_i}^3=0$ and ${X_i}^2g(X_i)=0$.
				\end{rem}

 \textbf{Acknowledgements}. The author thanks J. Bra\v{c}i\v{c} for helpful discussions.

\bibliographystyle{plain}

\end{document}